\documentclass[11pt]{article}

\usepackage{amsmath, amsthm, amsfonts}
\usepackage[utf8]{inputenc}	 

\newtheorem{thm}{Theorem}
\newtheorem*{thm*}{Implicit Function Theorem}
\newtheorem{cor}[thm]{Corollary}

\theoremstyle{definition}
\newtheorem{defn}{Definition}[section]
\theoremstyle{remark}

\def\R{\mathbb R}
\def\N{\mathbb N}

\title{Monochromatic metrics are generalized Berwald} 
\author{Nina Bartelmeß\thanks{Supported by DFG via Research Training Group 1523/2}$~$ and Vladimir S. Matveev}
  \date{}

\begin{document}
\maketitle

\abstract{We show that monochromatic Finsler metrics, i.e.,  Finsler metrics such that each two tangent spaces are isomorphic as normed spaces, are generalized Berwald metrics,  i.e., there exists an affine connection, possibly with torsion, that  preserves  the Finsler function.}

\tableofcontents


\section{Introduction}
\subsection{Definitions and main result}
A \textit{Finsler metric} on a smooth manifold $M$ of dimension $n \geq 2$ is a  function $F \colon TM \to [0,\infty)$ such that for every point $x \in M$ the restriction $F_x = F\vert_{T_x M}$ is a Minkowski norm. That means that 
\begin{enumerate}
   \item $F_x(\lambda \xi)=\lambda F_x(\xi)$ for all $\lambda \ge 0, \xi \in T_xM$
   \item $F_x(\xi+\eta)\le F_x(\xi)+ F_x(\eta)$ for all $\xi,\eta \in T_xM$
   \item $F_x(\xi)=0 \Rightarrow  \xi=0$
\end{enumerate}
We allow irreversibility and do not require   strict convexity. We assume that $F$ is smooth on the slit tangent bundle $TM\setminus \{0\}$. 

\begin{defn}
A Finsler metric $F$ on a connected manifold $M$ is said to be a \textit{generalized Berwald metric} if there exists a smooth affine connection $\nabla$ on $M$, called \textit{associated connection}, whose parallel transport preserves the Finsler function $F$. 
\end{defn}

That is, for every $x,y \in M$ and for any curve $\gamma \colon [0,1]\to M$ with $\gamma(0)=x$ and $\gamma(1)=y$ the parallel transport $P^\nabla_{\gamma}\colon T_x M \to T_y M$ along this curve is an isomorphism of  the normed spaces $(T_xM, F_x)$  and  $(T_yM, F_y)$ in the sense 
\begin{align*}
F_y(P^\nabla_{\gamma}(\xi))= F_x(\xi) \text{ for all } \xi \in T_x M.
\end{align*}
In the definition above, the connection may have a torsion. 

\begin{defn}  \label{Defmono}
A Finsler metric $F$ is called \textit{monochromatic} if for every two points $x,y \in M$  there exists a linear isomorphism between the tangent spaces at these points which is an isometry with respect to $F_x$ and $F_y$.
\end{defn}
 Clearly, as one of these points one can take some fixed point $x_0$, so monochromacy of a Finsler metric is equivalent to the existence of a field of linear isomorphisms  $A_x := A(x) \colon T_{x_0}M \to T_xM$ such that 
\begin{align*}
F(x,A_x(\xi)) = F(x_0,\xi) \text{ for all } \xi \in T_{x_0}M.
\end{align*}
We do not assume that  $A_x$ depends smoothly or even continuously on $x$.
This definition  is due to  David Bao \cite{Bao07}  and is motivated by a suggestion of Zhongmin Shen, who proposed to assign a unique color to each Minkowski norm. Generic Finsler spaces are then ``multicolored''  because different points of the manifold will generically correspond to different colors. Monochromatic manifolds are such that all points correspond to the same color. Our main result is: 

\begin{thm}\label{ThmMainResult}
Let $(M,F)$ be a connected Finsler manifold. Then,  $F$ is a generalized Berwald metric if and only if $F$ is monochromatic.
\end{thm}

In Theorem \ref{ThmMainResult} we assume that $M$ is at least of class $C^k$ and $F$ is at least of class $C^{k-1}$, $k \geq 2$. From the proof it will be clear that the associated connection will be at least of class $C^{k-2}$.
At the end of the paper, in \S \ref{dim3}, we discuss  the existence of non-Riemannian generalized Berwald metrics on closed manifolds of small dimensions.


\subsection{History and motivation}\label{subsecHistory}

 Definitions obviously similar to the definition of monochromatic metrics appeared many times in the literature, possibly first time in a commentary of Hermann Weyl on Riemann's habilitation address \cite{Riemann2013}. There, Weyl  suggested to consider Finsler manifolds such that all tangent spaces are isomorphic as normed spaces. It is not clear though whether he assumed that the field of isomorphisms $A_x$ in definition \ref{Defmono} depends smoothly on $x$.

In 1965 Detlef Laugwitz \cite{Laugwitz65} referred to Weyl's idea and suggested the following definition: 
he called a Finsler metric \textit{metrically homogeneous}  if, in our terminology, it is monochromatic and if in addition the field of linear isomorphisms $A_x$ in our definition of a monochromatic metric depends smoothly on $x$.  
An equivalent definition was given  by Yoshimiro Ichijyo \cite{Ichijyo76}  who called such Finsler metrics \textit{Finsler metrics modeled on a Minkowski space.}  Other equivalent definitions exist in the literature; such Finsler metrics were called \textit{1-form metrics} in e.g. \cite{Shimada} and   \textit{affine deformations of Minkowski spaces} in e.g. \cite{Tamassy}. 

It is easy to show and was independently done in \cite[Exercise 15.4.1]{Laugwitz65}, \cite[Theorem 2]{Ichijyo76} and, quite recently, in \cite[Theorem 1]{Tamassy}, that (locally) such metrics are generalized Berwald metrics. We essentially repeat their  proof at the end  of the proof of Theorem \ref{ThmMainResult}.

Many special cases of Theorem \ref{ThmMainResult} were proved before. For example,   \cite{Aradi} proved  it for left-invariant Finsler metrics  and \cite{Tayebi} proves  it for $(\alpha, \beta)$-metrics such that the $\alpha$-norm of $\beta$ is constant (in this paper it  is  also assumed that the metric should satisfy the so called sign property, but actually by Theorem \ref{ThmMainResult} this additional assumption can be omitted), see also \cite{Vincze1}. 
    
Recent interest to generalized Berwald spaces and monochromatic metrics is due to their relation to the Landsberg unicorn problem, see e.g. \cite{Vincze, Xu}, and  because for these metrics one obtains relatively simple formulas for different curvature-type invariants, see e.g. \cite{Aikou,Minguzzi,Libing,Shimada}. 
    
Recently, a 2-dimensional version of our Theorem \ref{ThmMainResult} was independently proved and applied  in \cite{ivanov}.

\subsection*{Acknowledgements} 
We thank S. Ivanov, C. Vincze and  J. Szilasi  for useful discussions. 


\section{Proof of Theorem \ref{ThmMainResult}}

The direction ``$\Rightarrow$'' is easy and was done many times before, in particular in the above mentioned  \cite{Ichijyo76,Laugwitz65}. Indeed, 
if  $F$ is a generalized Berwald metric on a connected manifold, then for every two points $x,y \in M$ the parallel transport $P^\nabla$ along any curve connecting $x$ and
$y$  gives us an isomorphism of $(T_xM, F_x)$ and $(T_yM, F_y)$.

The proof  in the direction "$\Leftarrow$" goes as follows: we first  show that, though we do not require \textit{a priori} that $A_x$ depends smoothly on $x$, one  can locally choose it such that it depends smoothly on $x$. As we explained before, the  case when $A_x$ depends smoothly on $x$ was solved before by Laugwitz \cite{Laugwitz65} and Ichiyjo \cite{Ichijyo76}. To make this work self-contained we repeat their proof. The last step is the transition from local to global, it uses a standard trick using the partition of unity argument.


\subsection{Locally, one can choose $A_x$ such that it depends smoothly on $x$.} 

Let $F$ be a smooth Finsler metric. We assume that it is not a Riemannian metric; the case of Riemannian metrics is trivial since they are automatically Berwald. We will work in a sufficiently small neighborhood of a point $p$, let $x=(x_1,...,x_n)$ be a local coordinate system in this neighborhood. 

We consider the Riemannian Binet-Legendre metric $g=g_F$ corresponding to $F$. The definition and properties of $g_F$ are  in \cite{Matveev1}. We will work in a local orthonormal frame   $e_1(x),..., e_n(x)$ with respect to this metric (i.e., $g_F(e_i, e_j)= \delta_{ij})$. The coordinates of tangent vectors will be coordinates in this frame.  The local existence of such an  orthonormal frame is known and  immediately follows from the  Gram-Schmidt orthogonalization process.

Next, we consider the Minkowski norm $F_p$ on $T_pM$, $m\in \mathbb{N}$ and vectors $ \xi_1,..., \xi_m\in T_pM$  such that the following conditions hold:  
\begin{itemize}
	\item [(I)] The differentials of the functions  $\psi_1,...,\psi_m \colon SO(n) \to  \R$,  $\psi_i(B)= F_p(B \xi_i)$  are linearly independent at $B= \textrm{Id}$. 
	\item  [(II)] The number $m$ is a maximal  number with the above property. 
\end{itemize} 
The notation $B \xi_i $ simply means the multiplication of a matrix $B\in SO(n) $ with the vector $ (  \xi_i^1,...,   \xi_i^n) \in \mathbb{R}^n$, where $  \xi_i^1,...,   \xi_i^n$ are coordinates of $ \xi_i$ in the orthonormal frame  $e_1(p),..., e_n(p)$. The resulting element of $\mathbb{R}^n$ will be identified with a  vector  of $T_pM$ (later also with a vector of $T_xM$)   via the basis $e_1(p),...,e_n(p)$ (later, via the basis $e_1(x),...,e_n(x)$). 

The existence of  such a number $m$ and the vectors $  \xi_1,...,   \xi_m$ is trivial. Indeed, $m$ such that (I) holds is bounded from above by $\frac{n(n+1)}{2}$ and because our Finsler metric is not a Riemannian metric the function $ B\mapsto  F_p(B  \xi )$ is not a constant (for each fixed $\xi \ne 0$) which implies the existence of at least one $\xi$ with property (I).  

It is  clear that the number $\frac{n(n+1)}{2}- m= \textrm{dim} \left(SO(n)\right)- m$ is the dimension of the group of endomorphisms  of $T_pM$ preserving $F_p$.  Let us  consider a local coordinate system $B=(b_1,..., b_{\frac{n(n+1)}2})$ on $SO(n)$ in a small neighborhood of the neutral element $\textrm{Id}$. In this  coordinates, $d\psi_i$ is simply the $\frac{n(n+1)}{2}$-tuple $\left(\tfrac{\partial \psi_i}{\partial b_1}, ..., \tfrac{\partial \psi_i}{\partial b_{\frac{n(n+1)}{2}}}\right)$ and the condition that the differentials  of $\psi_i$ are linearly independent means that the matrix 
\begin{equation} \label{1}
\left ( \frac{\partial \psi_i}{\partial b_j} \right)_{i=1,...,m; j= 1,..., \frac{n(n+1)}2}
\end{equation}
has rank $m$. Without loss of generality we will assume that the coordinates $b_1,...,b_{\frac{n(n+1)}{2}} $
are chosen such that the last $m$ columns of the matrix \eqref{1} form a nondegenerate matrix. 

\vspace{1ex} 
Let us now show that, possibly in a smaller neighborhood of $p$, one can choose the field $A_x$ in definition \ref{Defmono}  such that it depends smoothly on $x$. 
We  will use the Implicit Function Theorem; though it is  well known, we formulate it below  in order to fix the terminology.
\begin{thm*} Consider $W\subseteq \R^{k+m}$ with coordinates $(X,Y)= (X_1,...,X_k, Y_1,...,Y_m)$.  
Let $\Phi \colon \R^{k+m} \to \R^m$ be a continuously differentiable mapping  and $(\hat  X,\hat  Y) \in \R^{k+m}$ a point with $\Phi(\hat  X,\hat Y)=c$, where $c\in \R^m$. If  the $m\times m$-matrix 
\begin{align*}
 \left(\left(\frac{\partial \Phi_i}{\partial y_j}\right)_{i,j=1,...,m}\right)_{(\hat X, \hat Y)} \neq 0
\end{align*}
is nondegenerate, then there exists a non-empty open set $U\subset  \R^k$ containing $\hat X$, an open set $V\subset \R^m$ containing $\hat Y$ and a unique continuously differentiable mapping $g \colon U \to V$ such that $\Phi(X,g(X))=c$ for all $X\in U$ and such that for all points $(X,Y)\in U\times V $ with $\Phi(X,Y) =c$ we have that $Y= g(X)$.
Moreover, if $\Phi$ is of class $C^\ell$, $\ell\ge 1$, then $g$ is also of class $C^\ell$. \end{thm*}

Let us now apply this theorem to our situation. Set $k=n + \left(  \frac{n(n+1)}2- m \right)$.   The first $n$ coordinates of $X=(X_1,...,X_k)$ will be denoted by $x=(x_1,...,x_n)$, one should think about them as about local coordinates in a neighborhood of $p$. The remaining $\left(  \frac{n(n+1)}2- m \right)$ coordinates of $X=(X_1,...,X_k)$ will be denoted by $B_1=\left(b_1,...,b_{\frac{n(n+1)}2- m}\right)$, one can think about them as about the first portion of the local  coordinates on $SO(n)$ as discussed above. The coordinates $y=(y_1,...,y_m)=(B_2)$   should be viewed as the remaining portion of the  local  coordinates on $SO(n)$ as discussed above. 
   
Next, consider the mapping $\Phi \colon \R^{k+m} \to \R^m$, whose $j$-th component is given by 
\begin{align*}
\Phi_j(x,B)=F(x,B\xi_j). 
\end{align*}
The vectors  $\xi_j$ are precisely the vectors $\xi_1,...,\xi_m$ described above.
As $(\hat X,\hat Y)$ we take the point corresponding to $(x_1,...,x_n)= p$, $B= \mathrm{Id}$.

The differential of the mapping $\Phi$ at $(\hat X,\hat Y)$ is given by the following $m\times k$-matrix in which all partial derivatives are taken at the point $(X,Y)= (\hat X, \hat Y)$:  
\begin{align*}
d\Phi |_{\substack{B=\text{Id}\\x=p}} = \left( \begin{array}{c|c|c }
    \begin{matrix} &\\ \frac{\partial F(x,B\xi)}{\partial x} \\ &\end{matrix} &    \begin{matrix} &\\ \frac{\partial F(x,B\xi)}{\partial B_1}\\&  \end{matrix} &
      \begin{matrix} &\\ \frac{\partial F(x,B\xi)}{\partial B_2}\\  &\end{matrix}  
\end{array}\right). 
\end{align*}
By construction the last $m$ columns of the matrix form a nondegenerate $m\times m$ matrix, which is precisely the nondegenerate submatrix of matrix \eqref{1}.  Thus, all assumptions of the Implicit Function Theorem are satisfied and therefore there exists the 
smooth mapping  
\begin{align*}
B_2(x,B_1) 
\end{align*}
such that for $j=1,..,m$ holds
\begin{equation} \label{2}
\Phi(x,B_1,B_2(x,B_1))=\Phi(\hat X, \hat Y).
\end{equation}

Next, we construct a family $B_x\in SO(n)$ which depends smoothly on $x$ and which should be viewed as  a field of isomorphisms  $B(x)= B_x \colon T_{p}M \to T_{x}M$. In the local coordinates $B$ this family is given by  
\begin{align}\label{isomorphisms} B_x :=(B_1,B_2(x,B_1)).\end{align} 
Here $B_1$ is composed of the first ${\frac{n(n+1)}2- m}$ components of $\textrm{Id}$. By construction of $B_x$, we have 
\begin{align}
\label{InvBed}
F(x,B_x (\xi_i))= F(p, \xi_i)  \text{ for }   i=1,...,m.
\end{align} 
Our next goal is to  prove that (\ref{InvBed}) holds for all $\xi \in T_{p}M$. In order to do this, we  consider the following two  subsets of $SO(n)$: 
\begin{align*}
 U&=\{ u \in SO(n) ~ | ~ \forall \xi \in T_{p}M: ~ F(p,\xi)=F(p,u(\xi))\},\\
 U^{\prime}&=\{ u \in SO(n) ~ | ~ F(p,\xi_i)=F(p,u(\xi_i)) \text{ for }i=1,...,m\}.
\end{align*} 
Both subsets are compact,  $U$ is a Lie subgroup of $SO(n)$  and  we  have $U \subseteq  U^{\prime}$.  Let us now show that $U^{\prime} \setminus U$ is also compact. It is sufficient to show that $U^{\prime} \setminus U$ is closed, i.e., we need to show the  nonexistence of a sequence $u_1,...,u_i,...\in U^{\prime} \setminus U $ such that it converges to  an element of $U $.
 
We will use that the elements   $u\in U^{\prime}$ such that they are sufficiently close to  $\textrm{Id}\in SO(n)$ automatically lie in $U $.  Indeed, in a small neighborhood of  $\textrm{Id}$ 
both $U$ and $U^{\prime}$ are submanifolds of $SO(n)$ of the same dimension $\frac{n(n+1)}{2}- m$ and $U \subseteq  U^{\prime}$. 
  
Suppose a sequence $\{u_l\}_{l\in \N}\in U^{\prime}  \setminus U $  converges to $u\in U$. We consider the sequence $\{u^{-1}u_l\}_{l\in \N}$. It converges to $u^{-1}u =  \textrm{Id}$. Clearly, all elements of the sequence $\{u^{-1}u_l\}_{l\in \N}$ lie in $U^{\prime}$. Then, the elements of this sequence with sufficiently big indices  also lie in $U^\prime$. But this would imply that the corresponding  elements of the sequence $\{u_l\}_{l\in \N}$ lie  in $U$, which contradicts the assumption. Thus, $U^{\prime} \setminus U$ is compact.

\vspace{1ex} 
Now we can  show that for each $x\in M$  which is sufficiently close to $p$, the linear isomorphisms  $B_x$  constructed above  are  isomorphisms of the normed spaces $(T_{p}M, F_p)$ and $(T_{x}M, F_x)$.
Since   our  Finsler metric $F$ is monochromatic,  there exists    a linear isomorphism $A_x \colon T_{p}M \to T_xM$ such that
\begin{align*}
F(x,A_x(\xi)) = F(p,\xi) \text{ for all } \xi \in T_{p}M.
\end{align*}
Consider $A^{-1}_xB_x$ which lies in $U^{\prime}$ by  construction. In order  to show that (\ref{InvBed}) holds for $B_x$, it is sufficient  to show that $A^{-1}_xB_x \in U $. 
Let $\phi\colon SO(n) \to \R$ be the following function:
\begin{align*}
\phi(u) = \int_{K} |F(p,\xi) - F(p,u(\xi))| d\text{vol}_{g_F},
\end{align*}
where $K$ is the unit ball  and $\text{vol}_{g_F}$ the volume form of the Binet-Legendre metric $g_F$. The function  $\phi$ is continuous and nonnegative. 
Moreover,  $\phi(u) =0 $ if and only if $u \in U$. Then, because of 
 the compactness of $U^{\prime} \setminus U$,  there exists an $\epsilon > 0$ such that
\begin{align*}
\phi \big|_{U^\prime \setminus U} > \epsilon.
\end{align*}
Let us now consider $\phi(A^{-1}_xB_x)$:
\begin{align*}
\phi(A^{-1}_xB_x) & = \int_{K} |F(p,\xi) - F(p,A^{-1}_x B_x(\xi))| d\text{vol}_{g}\\
 &=\int_{K} |F(p,\xi) - F(p,B_x(\xi))| d\text{vol}_{g}\\
  &= \phi(B_x)
\end{align*} 
But $\phi(B_{p})=\phi(\text{Id})=0$, so that for $x$ sufficiently close to $p$ we have   $\phi(A^{-1}_xB_x)< \epsilon$, which implies that  $A^{-1}_xB_x \in U$ and hence $B_x \in U$.
Thus, locally one can find a field $B_x:T_pM\to T_xM$ 
of   isomorphisms  of the normed spaces $(T_pM, F_p)$ and  $(T_xM, F_x)$. 


\subsection{Construction of the associated connection }

Let us now show that our metric (still in a small neighborhood of the point $p$)  is  a generalized Berwald metric. Our proof  repeats, in  slightly different  notations, the proofs of Laugwitz \cite{Laugwitz65}  and Ichijyo \cite{Ichijyo76}.

Take a basis  $b_1 =\frac{\partial}{\partial x_1} ,...,b_n =\frac{\partial}{\partial x_n} $ on $T_pM$ and for each point $x$ of our small neighborhood 
consider the  basis $b_1(x)=B_xb_1,...,  b_n(x)=B_xb_n$. Let us now construct an  affine connection, possibly with torsion, such that its  parallel transport along any curve connecting $p$ and $x$  maps $ b_i(p) $ to $ b_i(x) $. Clearly, this condition on the connection  is equivalent to the condition 
\begin{align}
\label{Parallel}
\nabla   b_j(x) =0.
\end{align}
In coordinates,  (\ref{Parallel})  means that for all $i$ we have 
\begin{align*}
0 &=  \nabla_i b_k^j \\
&= \partial_i b(x)^j + \Gamma^j_{si} b(x)^s.
\end{align*}
By construction, in coordinates,  $b_1(x),...,b_n(x)$ are just   columns of the matrix $B(x)$. Next, consider   the matrices 
\begin{align*}
\Gamma_i &= \left(\begin{matrix}
\Gamma^1_{1i} & \dots & \Gamma^1_{ni} \\
\vdots & \ddots & \vdots\\
\Gamma^n_{1i} & \dots & \Gamma^n_{ni}
\end{matrix}\right).
\end{align*}
In this notation,  (\ref{Parallel}) reads 
\begin{align*}
\partial_i B(x) = -\Gamma_i B(x)
\end{align*}
and clearly has a  solution 
\begin{align*}
\Gamma_i = -\partial_i B(x) B(x)^{-1} . 
\end{align*}

This gives us  Christoffel symbols such that the parallel transport associated to this connection along any curve connecting $p$ and $x$ is given by $B_x$, and therefore is an isometry of the normed spaces $(T_pM, F_p)$ and $(T_xM, F_x)$.  We therefore proved the local version of  Theorem \ref{ThmMainResult}.

Finally we want to prove that there exists an  affine connection on the whole $M$ which preserves $F$. For each point $x\in U$,   consider a neighborhood $U(x)$ such that in this neighborhood there exists an affine  connection $\overset{x}{\nabla}$ on it which is an associated connection of $F$. Consider  a partition of unity $f_x$  subordinated  to  this  open cover, its existence is standard. Next, set $\nabla = \sum_x  f_x  \overset{x}{\nabla}$ (though the set of  points $x$ is infinite, at each sufficiently small neighborhood only finitely many terms in the sum are unequal to zero, guaranteed by the definition of the partition of unity; so the sum is well-defined). It is known and can easily be checked directly that this formula indeed defines an affine  connection; it is easy to check that the parallel transport in this connection preserves  our   Finlser metric $F$.


\section{Existence of generalized Berwald metrics on 2- and 3-dimensional closed manifolds} \label{dim3}

We start with dimension 2.  Clearly, the torus and the Klein bottle have non-Riemannian generalized Berwald metrics (induced by a Minkowski metric on the universal cover).
Let us show that the other closed 2-dimensional manifolds can not have non-Riemannian generalized Berwald metrics. In order to prove this, observe that  the group $SO(2)$ is one-dimensional and thereby each of its proper connected subgroup is discrete. Then, the local holonomy group of the associated affine connection of a non-Riemannian generalized Berwald metric is trivial, which implies,  by the Ambrose-Singer Theorem,  that the associated  affine connection is flat. Then, by \cite{Milnor}, the surface has Euler characteristic equal to zero.\\

\noindent
In higher dimension we can prove the following:
\begin{thm}
Every closed manifold with Euler characteristic zero admits a non-Riemannian generalized Berwald metric.
\end{thm}

\begin{proof}
It is known that on manifolds with Euler characteristic zero there exists a non-vanishing vector field $V$. Take now a Riemannian metric $g$ on $M$ and normalize $V$ to the length $\frac{1}{2}$ with respect to $g$.
Consider now the Randers metric $F_R$ on $T_xM$ generated by $g$ and $V$, i.e., the unit balls of $F_R$ are the $V$-translations of the unit balls of $g$. Those convex balls in every tangent space $T_xM$ are isomorphic to each other. By our Theorem \ref{ThmMainResult} we obtain a generalized Berwald metric. 
\end{proof}

\begin{cor}
In dimension 3 every closed manifold has Euler characteristic zero and thus admits a non-Riemannian generalized Berwald metric.
\end{cor}
\noindent
It is an interesting problem to understand in all dimensions what manifolds can admit non-Riemannian generalized Berwald metrics.


  Institute of Mathematics, FSU Jena, 07737 Jena Germany,\\  nina.bartelmess@uni-jena.de,  vladimir.matveev@uni-jena.de 
\end{document}